\documentclass[12pt,titlepage]{article}
\usepackage{rsipacks}

\newtheorem{theorem}{Theorem}[section]
\newtheorem{coro}{Corollary}[section]
\newtheorem{conj}{Conjecture}[section]
\theoremstyle{definition}
\newtheorem{defin}{Definition}[section]

\begin{document}



\title{
Emptying Sets: The Cookie Monster Problem
}

\author{
Megan Belzner
\vspace{0.5in}\\
under the direction of\\
Wuttisak Trongsiriwat\\
Massachusetts Institute of Technology
\vspace{1in}
}


\date{
Research Science Institute\\
\rsifinalpaperdate
}

\begin{singlespace}
\maketitle
\end{singlespace}

\begin{abstract}

Given a set of integers $S = \{k_1,\ k_2,\ \ldots,\ k_n\}$, the Cookie Monster Problem is the problem of making all elements of the set equal 0 in the minimum number of moves. Consider the analogy of cookie jars with distinct numbers of cookies, such that $k_i$ is the number of cookies in the $i$th jar. The ``Cookie Monster'' wants to eat all the cookies, but at each move he must choose some subset of the jars and eat the same amount from each jar. The \emph{Cookie Monster Number of $S$}, $CM(S)$, is the minimum number of such moves necessary to empty the jars. It has been shown previously that $\lceil \log_2(|S|+1) \rceil \leq CM(S) \leq |S|$. In this paper we classify sets by determining what conditions are necessary for $CM(S)$ to equal 2 or 3 and what effect certain restrictions have on $CM(S)$. We also provide an alternative interpretation of the problem in the form of a combinatorial game and analyze the losing positions.

\vspace{1in}
\begin{center}\textbf{Summary}\end{center}

Given some number of cookie jars, each with a distinct number of cookies, the Cookie Monster Problem asks us to empty the jars in the fewest number of moves possible. One move consists of choosing some number of the jars and taking the same number of cookies from each. In this paper we examine how different restrictions on the set of cookie jars affect the minimum number of moves necessary to empty the jars.

\end{abstract}


\section{Introduction}

The \emph{Cookie Monster Problem} (CMP) is about emptying a given set in the fewest number of moves---that is, making every element in the set equal zero. First proposed as a simple puzzle in 2002 in the book \emph{The Inquisitive Problem Solver} \cite{quiz}, it has since been analyzed and expanded by Michael Cavers \cite{lecture}. The initial formulation presents a set of $n = 15$ cookie jars with $i$ cookies in the $i$th jar. The ``Cookie Monster'' wants to eat all of the cookies, but he has to do so in a series of \emph{moves}. If one move consists of taking some subset of the jars and eating the same number of cookies from each jar, the CMP asks how to empty all jars in fewer than five steps. In this case, the optimal solution consists of removing 8 cookies from every jar with at least that many, then 4 cookies, then 2, then 1. This solution can be represented as $\langle 8,\ 4,\ 2,\ 1\rangle$. After each step, jars with the same number of cookies can be treated as the same jar, since depleting the jars at different rates does not help \cite{quiz}.

A set of cookie jars, $S$, can be represented in the form $S = \{k_1,\ k_2,\ \ldots,\ k_n\}$ where $k_i$ is the number of cookies in the $i$th jar and each consecutive step is represented similarly with all equal elements treated as one element and all zeros dropped. Arrows ($\rightarrow$) between sets denote a single move.
$$\{15,\ 14,\ \ldots,\ 2,\ 1\} \rightarrow \{7,\ 6,\ \ldots,\ 2,\ 1\} \rightarrow \{3,\ 2,\ 1\} \rightarrow \{1\} \rightarrow \{\}.$$

Several variations on the greedy algorithm, which seeks to maximize a given parameter, have been proposed, though none of them give the optimal solution in all cases \cite{quiz}:

\begin{itemize}
\item The \emph{Empty the Most Jars Algorithm} (EMJA) suggests that one should reduce the functional number of jars as much as possible each move.
\begin{itemize}
\item Given $\{15,\ 13,\ 12,\ 4,\ 2,\ 1\}$, the first move could be to take 11 cookies from each of the first three jars, leaving the set $\{4,\ 2,\ 1\}$. Thus, in effect, three jars have been emptied.
\end{itemize}
\item The \emph{Take the Most Cookies Algorithm} (TMCA) takes as many cookies as possible at each step.
\begin{itemize}
\item With $\{15,\ 13,\ 12,\ 4,\ 2,\ 1\}$, the first step would be to take 12 cookies from each of the first three jars, leaving $\{4,\ 3,\ 2,\ 1\}$.
\end{itemize}
\item The \emph{Binary Algorithm} (BA) finds $x$ as large as possible and takes $2^x$ cookies from each jar that contains at least that many.
\begin{itemize}
\item From $\{15,\ 13,\ 12,\ 4,\ 2,\ 1\}$, the first step would take $2^3 = 8$ cookies from the first three jars, leaving $\{7,\ 5,\ 4,\ 2,\ 1\}$.
\end{itemize}
\end{itemize}

Section \ref{found} begins with definitions and an overview of previous work on this problem. In Section \ref{clsets}, all sets where $|S| = 3$ or $CM(S) = 3$ are classified, and some additional bounds are proposed dependent on certain conditions. Section \ref{spec} examines the properties of $CM(S)$ for arithmetic and geometric sequences and the Fibonacci sequence. Finally, Section \ref{game} provides a version of the problem as a combinatorial game and analyzes properties of the losing positions of a special case of this game.

\section{Foundations}\label{found}

The original CMP has been formalized and expanded by Cavers \cite{lecture}.

\begin{defin}\label{subs}
Given a set $S = \{k_1,\ k_2,\ \ldots,\ k_n\}$ of distinct integers, the \emph{Cookie Monster Number of S}, $CM(S)$, is the minimum number of moves required to make all elements of $S$ equal zero. Alternatively, given a multiset $A = \langle a_1,\ a_2,\ \ldots,\ a_m\rangle$, let the set $A^+$ be the set of the sums of all possible subsets of the elements in $A$. Then, for a set $S\subseteq A^+$, $A$ provides a series of numbers to remove from $S$ to empty the set $S$ \cite{lecture}. Thus, $CM(S)$ is equal to the size of the smallest multiset $A$ such that $S\subseteq A^+$ \cite{lecture}.
\end{defin}

For example, let $S = \{13,\ 10,\ 7,\ 6\}$ and suppose that $A = \langle 7,\ 3,\ 3\rangle$. Then, $A^+ = \{a_1+a_2+a_3,\ a_1+a_2,\ a_1,\ a_2+a_3,\ a_2\} = \{13,\ 10,\ 7,\ 6,\ 3\}$ which satisfies $S\subseteq A^+$. Thus, one way to empty the set $S$ first removes 7 from $k_1,k_2,k_3$, then 3 from $k_1,k_2,k_4$, then 3 from the remaining elements $k_1,k_4$:
$$\{13,\ 10,\ 7,\ 6\} \rightarrow \{6,\ 3\} \rightarrow \{3\} \rightarrow \{\}.$$

This formulation provides upper and lower bounds on the value of $CM(S)$ for a set of size $n$, first suggested and proven by Cavers \cite{lecture}.

\begin{theorem}\label{bounds}
Given a set $S = \{k_1,\ k_2,\ \ldots,\ k_n\}$:
$$\lceil \log_2(n+1)\rceil \leq CM(S) \leq n.$$
\end{theorem}

\begin{proof}
The upper bound is taken from the algorithm ``\emph{at step $i$ empty jar $i$}'', which empties any set in $n$ steps \cite{lecture,blog}.

For the lower bound, consider a multiset $A = \langle a_1,\ a_2,\ \ldots,\ a_m\rangle$ such that $A$ is the minimum-size multiset for which $S\subseteq A^+$. It then follows that $CM(S) = m$. The set $A^+$ has at most $\sum_{i=1}^m {m \choose i} = 2^m-1$ entries, and thus, $n \leq 2^m-1$. Then, $m \geq \log_2(n+1)$ and the result follows \cite{lecture}.
\end{proof}

There has been work towards characterizing sets for which $CM(S)$ is equal to the upper or lower bounds given in Theorem \ref{bounds}. Several such cases are known: $CM(S)$ for any set $S$ which is an arithmetic sequence of the form $k_i = ai$ such as $S = \{n,\ldots,2,1\}$ is exactly equal to the lower bound, and $CM(S)$ for any set $S$ which is a geometric sequence such as $S = \{2^{n-1},\ \ldots,\ 2^1,\ 2^0\}$ is exactly equal to the upper bound. However, general characterizations do not currently exist for sets of size $n \geq 4$.

\section{Classifying Sets}\label{clsets}

Although $CM(S)$ is known for many types of sets, such as arithmetic and geometric progressions (as mentioned above), little work has been completed towards classifying $CM(S)$ in general. The simplest non-trivial case, $|S| = 3$, is analyzed, and it gives insight into sets where $CM(S) = 3$. Additionally, some general rules can be stated about sets and their properties.

\subsection{Sets of Size Three}

For a set $S$ of size 3, the upper and lower bounds on $CM(S)$ are established by Theorem \ref{bounds} to be 3 and 2, respectively. $CM(S)$ only equals 2 if certain conditions are satisfied.

\begin{theorem}\label{three}
Consider a set $S = \{k_1,\ k_2,\ k_3\}$ such that $0 < k_1 < k_2 < k_3$. Then, $CM(S) = 2$ if and only if $k_3 = k_1 + k_2$.
\end{theorem}

\begin{proof}
If $k_3 = k_1 + k_2$, then a procedure for emptying the set in 2 moves consists of removing $k_1$ from $k_1$ and $k_3$, then removing $k_2$ from the remaining element(s):
$$\{k_1,\ k_2,\ k_1 + k_2\} \rightarrow \{k_2\} \rightarrow \{\}.$$

Conversely, suppose that $CM(S) = 2$. By Definition \ref{subs}, there exists a multiset $A$ with $|A| = 2$ and $S \subseteq A^+$. If $A = \langle a_1,\ a_2\rangle$, then the set $A^+ = \{a_1,\ a_2,\ a_1+a_2\}$, and the only way a set $S$ such that $|S| = 3$ can be a subset of $A^+$ is if $S = A^+$. Thus, $k_3 = k_1 + k_2$.
\end{proof}

\subsection{Classification of $S$ with $CM(S) = 3$}\label{class}

Consider the multiset $A = \langle a_1,\ a_2,\ a_3\rangle$. This provides a solution for any set $S$ such that $A$ is a minimum size multiset (that is, $CM(S) = 3$) with $S \subseteq A^+ = \{a_1,\ a_2,\ a_3,\ a_1+a_2,\ a_1+a_3,\ a_2+a_3,\ a_1+a_2+a_3\}$. For any $n$, all possible sets for which $CM(S) = 3$ can be generated by taking every possible subset of $A^+$. Some such sets for $n = 4$ are:

\begin{center}
$\{a_1,\ a_2,\ a_1+a_2,\ a_1+a_3\}$,

$\{a_1,\ a_2,\ a_3,\ a_1+a_2+a_3\}$,

$\{a_1,\ a_3,\ a_1+a_2,\ a_2+a_3\}$,

$\{a_1,\ a_1+a_2,\ a_1+a_3,\ a_2+a_3\}$,

$\{a_1+a_2,\ a_1+a_3,\ a_2+a_3,\ a_1+a_2+a_3\}$.
\end{center}

However, these sets are all given in terms of the elements of $A$. Given a set $S$, $A$ is not yet known, and so alternative representations in terms of the elements of $S$ can be used to find $A$. Examples where each set above is in the form $\{k_1,\ k_2,\ k_3,\ k_4\}$ are given in Table \ref{alt4}---these equalities hold without any knowledge of $A$, and are thus useful for classification. \\

\begin{table}[hbpt]
\begin{center}
\begin{tabular}{|c|c|} \hline
$A$-level set & $S$-level equation \\ \hline
$\{a_1,\ a_2,\ a_1+a_2,\ a_1+a_3\}$ & $k_1 + k_2 = k_3$ \\
$\{a_1,\ a_2,\ a_3,\ a_1+a_2+a_3\}$ & $k_1 + k_2 + k_3 = k_4$ \\
$\{a_1,\ a_3,\ a_1+a_2,\ a_2+a_3\}$ & $k_1 + k_4 = k_2 + k_3$ \\
$\{a_1,\ a_1+a_2,\ a_1+a_3,\ a_2+a_3\}$ & $2k_1 + k_4 = k_2 + k_3$ \\
$\{a_1+a_2,\ a_1+a_3,\ a_2+a_3,\ a_1+a_2+a_3\}$ & $k_1 + k_2 + k_3 = 2k_4$ \\ \hline
\end{tabular}
\caption{Sample equalities for $CM(S) = 3$ where $|S| = 4$}\label{alt4}
\end{center}
\end{table}

In fact, all sets of size 4 where $CM(S) = 3$ are covered under these equalities, as from the full list it can be seen that any other set arises via permutation of $a_1,\ a_2,\ a_3$ or $a_1+a_2,\ a_1+a_3,\ a_2+a_3$ in one of these sets. Other values for $|S|$ can be represented similarly in terms of equations with the elements of $S$. However, while size 4 sets only need one equation in order to cover all possibilities, other sets need systems of equations with the number of equations equal to $|S| - 3$. The full list of equations is given in Appendix \ref{sets}. An exhaustive test can be used to confirm that any set which satisfies the $A$-level representation also satisfies the $S$-level one, and vice versa. Each $S$-level representation is constructed from the $A$-level version and is thus equivalent, and each $S$-level version can be individually shown to be solvable in 3 moves.


\subsection{General Properties}

In addition to the specific sets given above, other rules about the relationship between $CM(S)$ and the set $S$ can be determined regarding any subsets summing to the same value.

\begin{theorem}
Consider a set $S$ and let $n = |S|$. If two disjoint subsets of $S$ sum to the same value, then $CM(S) \leq n - 1$.
\end{theorem}

\begin{proof}
Suppose $\{b_1,\ \ldots,\ b_r,\ c_1,\ \ldots,\ c_s\} \subseteq S$ and $b_1 + \ldots + b_r = c_1 + \ldots + c_s$. Any other values which are outside of the equality can be ignored, as they will take at most one additional move to remove. If either $r = 1$ or $s = 1$, then $S$ can be emptied in $r + s - 1$ moves in a manner similar to that used in the proof of Theorem \ref{three}. Suppose, however, that $r,s > 1$. The base case of $r + s = 4$ gives the non-trivial possibility of $r = s = 2$, that is, $b_1 + b_2 = c_1 + c_2$. Supposing that $b_1 > c_1$, this equation can be reordered into $c_2 = (b_1 - c_1) + b_2$. Then, there is a way to empty the set in $r + s - 1 = 3$ moves:
$$\{b_1,\ b_2,\ c_1,\ (b_1 - c_1) + b_2\} \xrightarrow{-(b_1 - c_1)} \{b_2,\ c_1\} \xrightarrow{-c_1} \{b_2\} \xrightarrow{-b_2} \{\}.$$

Then, given any $r + s$, the set can be reduced to a set of size $r + s - 1$ in a single step. With any equality $b_1 + \ldots + b_r = c_1 + \ldots + c_s$, the smaller first value can be subtracted from the larger. This difference becomes a single element in a new set of size $r + s - 1$, for which there then exists a similar equality $(b_1 - c_1) + \ldots + b_r = c_2 + \ldots + c_s$. If the set of size $r + s - 1$ is solvable in one fewer steps than its size, so is the set of size $r + s$, and with the base case of $r + s = 4$, the result follows by induction.
\end{proof}

\begin{coro}\label{nosum}
If $CM(S) = n$, then no two disjoint subsets of $S$ sum to the same value.
\end{coro}

Note that the converse of Corollary \ref{nosum} is not true. For example, the set $\{5,\ 9,\ 12,\ 13\}$ can be emptied in three moves ($\{5,\ 9,\ 12,\ 13\} \rightarrow \{4,\ 8,\ 12\} \rightarrow \{8\} \rightarrow \{\}$) although no two subsets sum to the same value.

\begin{coro}
Consider a set $S$ with $x$ disjoint pairs of subsets that sum to the same value, such as $S = \{b_1,\ \ldots,\ b_r,\ c_1,\ \ldots,\ c_s,\ d_1,\ \ldots,\ d_t,\ e_1,\ \ldots,\ e_u\}$ such that $b_1 + \ldots + b_r = c_1 + \ldots + c_s$ and $d_1 + \ldots + d_t = e_1 + \ldots + e_u$. For $S$ of this form, $CM(S) \leq n - x$.
\end{coro}

\section{Special Sets}\label{spec}

In addition to the classification of sets as above, certain sequences give interesting values of $CM(S)$.

\subsection{Arithmetic Sequences}

Arithmetic sequences of the form $k_i = yi + z$ are initially interesting as the case where $y = 1$ and $z = 0$ is the simplest example of a set where $CM(S)$ is exactly equal to the lower bound (for any size set). It follows from this that any arithmetic sequence where $z = 0$ has the same property, as common factors make no difference.

However, when $z \neq 0$, $CM(S)$ is not quite as simple. In all cases, the set can be emptied in one more move than the case of $z = 0$ simply by subtracting $z$. However, for some $|S|$ (namely, any power of 2), $CM(S)$ is still equal to the lower bound.

\begin{theorem}
For $S$ with $n$ elements of the form $k_i = yi + z$ where $y,z \neq 0$, $CM(S) = \lceil \log_2n \rceil + 1$.
\end{theorem}

\begin{proof}
For any $n = 2^x$ for some value $x$, $\lceil \log_2n \rceil + 1 = \lceil \log_2(n+1) \rceil$, while when $n \neq 2^x$, $\lceil \log_2n \rceil + 1 = \lceil \log_2(n+1) \rceil + 1$. As stated above, any set such that $k_i = yi + z$ can be emptied in $\lceil \log_2(n+1) \rceil + 1$ moves by removing $z$ from each element at the first step and proceeding as with $z = 0$. Then, for any set to be able to be emptied in one fewer move, the number of elements must be halved at each move ($n_j = \lfloor \frac{n_{j-1}}{2} \rfloor$). Any even number of jars with elements in an arithmetic progression can be halved. For example, supposing $S = \{y+z,\ 2y+z,\ 3y+z,\ 4y+z,\ 5y+z,\ 6y+z\}$, subtracting $3y$ from each of the last three elements will leave three elements remaining. Supposing, however, that $n$ was odd, the minimum number of elements left would be $\lceil \frac{n}{2} \rceil$. The case where $z = 0$ uses the idea of emptying the median element entirely, but this is not possible here as the $z$ will still be left over. The only sets where $n$ is even at every step (except for the last, where there is only one element) are ones where $n = 2^x$, and the result follows.
\end{proof}

\subsection{Geometric Sequences}

Similar to arithmetic sequences being equal to the lower bound, geometric sequences provide the simplest example of sets for which $CM(S) = n$. This is true for any set of the form $k_i = wy^{i-1}$ with $y \geq 2$ and $w > 0$. However, this may no longer hold if some constant is added to each term in the sequence, though the set is also no longer geometric. Any set for which $k_2 > k_1,\ k_3 > k_2 + k_1,\ \cdots,\ k_n > k_{n-1} + \ldots + k_1$ is at the upper bound $CM(S) = n$, so the constant must be able to overcome this. In addition, if $k_1 + k_2 > k_3$, there is no guarantee that $S$ can be emptied in fewer than $n$ moves.

\subsection{The Fibonacci Sequence}

Another sequence which offers an interesting equation for $CM(S)$ is the Fibonacci sequence.

\begin{theorem}
For $S = \{F_2,\ \ldots,\ F_n\}$ where $F_i = F_{i-2} + F_{i-1}$ with $F_0 = 0$ and $F_1 = 1$, $CM(S) = \lceil \frac{n}{2} \rceil$.
\end{theorem}

\begin{proof}
The proof that such a set can be solved in $\frac{n}{2}$ moves follows from the nature of the Fibonacci sequence. At each step, two elements can be removed by subtracting the second largest element from the largest two, that is, subtracting $F_{i-1}$ from $F_{i-1}$ and $F_i$, which leaves 0 and $F_{i-2}$, respectively. This process can be repeated, until no elements remain (if $n$ is even) or one element remains (if $n$ is odd). For a proof of equality, consider the set $S = \{F_2,\ \ldots,\ F_n,\ F_{n+1},\ F_{n+2}\}$. For the Fibonacci numbers, $F_{n+2} = 1 + F_1 + F_2 + \ldots + F_n$, and so $F_{n+2} > F_1 + F_2 + \ldots + F_n$. Thus, at least one additional move is required to deal with $F_{n+2}$. Working backwards, the same can be said, and so every two elements will require one move.
\end{proof}

\section{The Cookie Monster Combinatorial Game}\label{game}

The Cookie Monster Problem can be altered to obtain a combinatorial game. In this game players alternate turns, choosing any subset of the jars and taking the same nonzero amount from each jar, with the aim of emptying the last jar. Though the premise is similar, the approach is fundamentally different as the aim is no longer necessarily to use the least number of moves---in fact, in many cases it is beneficial to lengthen the game. However, it can be determined who will win if both players play with perfect strategy by analyzing the losing positions of the game.

\subsection{Wythoff's Game}

A game with two jars has already been analyzed as ``Wythoff's Game'' \cite{wythoff}. The game is suggested as a variation of Nim, a game which consists of any number of jars where each move is to take some amount from any one jar. Here, instead of merely taking any amount from one jar the player can instead choose to take same amount from both jars. The losing positions can be defined as any combination of numbers $\{p,\ q\}$ where any move will lead to a winning position. Beginning with $p_1 = 1$, losing positions can be generated by taking $q_i = p_i + i$, and $p_i$ as the smallest integer which has not yet appeared in the set of losing positions \cite{wythoff}. The first ten Wythoff pairs are listed in Table \ref{pairs}.

\begin{table}[hbtp]
\begin{center}
\begin{tabular}{|ccc|} \hline
$i$ & $p_i$ & $q_i$ \\ \hline
1 & 1 & 2 \\
2 & 3 & 5 \\
3 & 4 & 7 \\
4 & 6 & 10 \\
5 & 8 & 13 \\
6 & 9 & 15 \\
7 & 11 & 18 \\
8 & 12 & 20 \\
9 & 14 & 23 \\
10 & 16 & 26 \\ \hline
\end{tabular}
\caption{The first ten Wythoff pairs}\label{pairs}
\end{center}
\end{table}

The construction of these pairs follows as such: beginning with $\{0,\ 0\}$, the base losing position as no moves can be made, the next losing position is seen to be $\{1,\ 2\}$ as this is the first position for which no move can reduce it to another losing position---namely, $\{0,\ 0\}$. Then, no other set can have the same difference between $p$ and $q$, that is, no other set can have $q_i - p_i = 1$, because this would be easily reducible to $\{1,\ 2\}$. So the next possible $p$ is taken to be 3, and the corresponding $q$ would be $p$ plus the smallest number which has not already appeared for some $q_i - p_i$, and the construction continues in this manner. This also shows that $\{p_i\}\sqcup\{q_i\} = \mathbb{N}$. It was later found \cite{wythoff} that these positions follow the equation $p_i = \lfloor i\phi \rfloor$ and $q_i = \lfloor i\phi^2 \rfloor$, where $\phi = \frac{1 + \sqrt{5}}{2}$.

\subsection{A Game With Three Jars}

The case of two jars can be generalized to a version of the game with three jars. For completeness, the Wythoff pairs will also be considered for this version in the form $\{0,\ p^0_i,\ q^0_i\}$. From here, sets of the form $\{1,\ p^1_i,\ q^1_i\}$ can be generated. However, the pattern is not as regular as the case of $\{0,\ p^0_i,\ q^0_i\}$. Although a similar method is used to generate the $p^1_i,q^1_i$ pairs, the existence of other losing positions of the form $\{0,\ p^0_i,\ q^0_i\}$ results in an irregular distribution of the difference $d_i = q^1_i - p^1_i$. Beginning with $\{1,\ 1,\ 4\}$ ($p^1_1 = 1, q^1_1 = 4$), the next $p^1_i$ is again the smallest positive integer which has not yet appeared in the set of $p^1_i,q^1_i$ pairs. However, 2 must be skipped in this determination as any set $\{1,\ 2,\ q^1_i\}$ or $\{1,\ p^1_i,\ 2\}$ would be reducible to the losing position $\{0,\ 1,\ 2\}$ in a single move by emptying all of the unknown jar. From here, the corresponding $q^1_i$ is not simply $p^1_i + i$. All possible differences $d_i$ (the set of positive integers $\geq 0$) will be used exactly once, but not necessarily in order as some may reduce easily to a set $\{0,\ p^1_i,\ q^1_i\}$. The first forty $p^1_i,q^1_i$ pairs along with their difference $d_i = q^1_i - p^1_i$ are given in Table \ref{xy}.

\begin{table}[hbtp]
\begin{center}
\begin{tabular}{|ccc|ccc|ccc|ccc|} \hline
$p^1_i$ & $q^1_i$ & $d_i$ & $p^1_i$ & $q^1_i$ & $d_i$ & $p^1_i$ & $q^1_i$ & $d_i$ & $p^1_i$ & $q^1_i$ & $d_i$ \\ \hline
1 & 4 & 3 & 18 & 27 & 9 & 34 & 53 & 19 & 50 & 79 & 29 \\
3 & 3 & 0 & 20 & 33 & 13 & 36 & 56 & 20 & 52 & 82 & 30 \\
5 & 6 & 1 & 21 & 32 & 11 & 37 & 59 & 22 & 54 & 88 & 34 \\
7 & 9 & 2 & 23 & 35 & 12 & 39 & 64 & 25 & 55 & 87 & 32 \\
8 & 12 & 4 & 24 & 38 & 14 & 40 & 63 & 23 & 57 & 90 & 33 \\
10 & 17 & 7 & 26 & 43 & 17 & 41 & 67 & 26 & 58 & 93 & 35 \\
11 & 16 & 5 & 28 & 46 & 18 & 42 & 66 & 24 & 60 & 98 & 38 \\
13 & 19 & 6 & 29 & 44 & 15 & 45 & 72 & 27 & 61 & 97 & 36 \\
14 & 22 & 8 & 30 & 51 & 21 & 48 & 76 & 28 & 62 & 101 & 39 \\
15 & 25 & 10 & 31 & 47 & 16 & 49 & 80 & 31 & 65 & 102 & 37 \\ \hline
\end{tabular}
\caption{The first forty $p^1_i,q^1_i$ pairs for $\{1,p^1_i,q^1_i\}$ losing positions}\label{xy}
\end{center}
\end{table}

The $p^1_i,q^1_i$ pairs have several similar properties to $p^0_i,q^0_i$ pairs. As mentioned above, every possible $d_i$ will appear exactly once. In addition, the disjoint union of $\{p^1_i\}$ and $\{q^1_i\}$ is very nearly the set of all natural numbers, much as the disjoint union of $\{p^0_i\}$ and $\{q^0_i\}$ is. However, it excludes the number 2. If $p^0_i$ and $p^1_i$ are graphed against $q^0_i$ and $q^1_i$, respectively, the graphs will follow almost identical patterns, though $p^1_i$ against $q^1_i$ has greater variation. Despite several correlations between $p^0_i,q^0_i$ pairs and $p^1_i,q^1_i$ pairs, the latter cannot be as easily matched to an equation as the former can.

\begin{conj}\label{slope}
On average, the graphs of $p^0_i$ against $q^0_i$ and $p^1_i$ against $q^1_i$ have the same slope, that is $\frac{q^0_i}{p^0_i} \approx \frac{q^1_i}{p^1_i}$.
\end{conj}

\begin{conj}\label{diff}
If the sets of all losing positions of the form $\{0,\ p^0_i,\ q^0_i\}$ and $\{1,\ p^1_i,\ q^1_i\}$, respectively, are arranged in order of increasing $p^0_i$ and $p^1_i$, there exist constant bounds on the values $p^1_i - p^0_i$ and $(q^1_i - p^1_i) - (q^0_i - p^0_i)$.
\end{conj}

Both of these conjectures hold for the first 100 $p^1_i,q^1_i$ pairs. Within the first 100 pairs, no difference $p^1_i - p^0_i$ is greater than 2 or less than $-1$ and no difference $(q^1_i - p^1_i) - (q^0_i - p^0_i)$ is greater than 3 or less than $-4$. However, with further $p^1_i,q^1_i$ pairs the maximum and minimum differences increase, though generally the differences remain close to 0.

\section{Conclusion}

For some sets $S$, $CM(S)$ can be easily determined. Sets where $|S| = 3$ have been fully classified, as have sets for which $CM(S) = 3$ for sets of size 3 (where any set for which $k_1 + k_2 = k_3$ does not hold has $CM(S) = 3$) up to size 7 (which is the largest set $S$ which can be a subset of $A^+$ if $|A| = 3$). In addition, arithmetic and geometric sequences and the Fibonacci sequence have been examined, and a version of the problem as a combinatorial game has been partially analyzed. Despite this work in classifying sets, few more general statements can be made for larger sets $S$ or $A$. Some of the properties for $CM(S) = 3$ may be able to be generalized, or other restrictions found. The combinatorial game could also be more fully analyzed. In addition, there is a possibility but no proof that this problem is NP-complete.

\section{Acknowledgments}

I would like to thank my mentor, Mr. Wuttisak Trongsiriwat, as well as Dr. Tanya Khovanova for their help in suggesting the problem and guiding my work. I would also like to thank my tutor, Dr. Jake Wildstrom, and the other RSI staff for helping me in various ways. Additionally, I would like to thank the Massachusetts Institute of Technology (MIT) for hosting the Research Science Institute. Finally, I would like to thank the Center for Excellence in Education and MIT Math Department as well as all of the sponsors of the Research Science Institute for organizing and giving me the opportunity to conduct this research.

\begin{singlespace}
 


\end{singlespace}

%

\appendix


\section{Sets Where $CM(S) = 3$}\label{sets}

What follows is a full list of both the $A$-level and corresponding $S$-level representations for all sets where $CM(S) = 3$ (see Section \ref{class}). Where multiple equations exist, a set $S$ must satisfy \emph{all} of the equations to match. Additionally, some sets have variations where elements of the $A$-level set are replaced by equivalent ones---$a_1,a_2,a_3$ can be replaced by any of the other two, and similarly with $a_1+a_2,a_1+a_3,a_2+a_3$ as long as each only appears at most once.

\begin{table}[htpb]
\begin{center}
\begin{tabular}{|c|c|} \hline
$A$-level set & $S$-level equation \\ \hline
$\{a_1,\ a_2,\ a_1+a_2,\ a_1+a_3\}$ & $k_1 + k_2 = k_3$ \\ \hline
$\{a_1,\ a_2,\ a_3,\ a_1+a_2+a_3\}$ & $k_1 + k_2 + k_3 = k_4$ \\ \hline
$\{a_1,\ a_3,\ a_1+a_2,\ a_2+a_3\}$ & $k_1 + k_4 = k_2 + k_3$ \\ \hline
$\{a_1,\ a_1+a_2,\ a_1+a_3,\ a_2+a_3\}$ & $2k_1 + k_4 = k_2 + k_3$ \\ \hline
$\{a_1+a_2,\ a_1+a_3,\ a_2+a_3,\ a_1+a_2+a_3\}$ & $k_1 + k_2 + k_3 = 2k_4$ \\ \hline
\end{tabular}
\caption{Sets and equations for $|S| = 4$}
\end{center}
\end{table}

\begin{table}[htpb]
\begin{center}
\begin{tabular}{|c|c|} \hline
$A$-level set & $S$-level equations \\ \hline
$\{a_1,\ a_2,\ a_3,\ a_1+a_3,\ a_2+a_3\}$ &
\begin{tabular}{c}
$k_1 + k_3 = k_4$ \\
$k_2 + k_3 = k_5$ \\
\end{tabular} \\ \hline
$\{a_1,\ a_2,\ a_3,\ a_1+a_3,\ a_1+a_2+a_3\}$ &
\begin{tabular}{c}
$k_1 + k_3 = k_4$ \\
$k_1 + k_2 + k_3 = k_5$ \\
\end{tabular} \\ \hline
$\{a_1,\ a_2,\ a_1+a_2,\ a_1+a_3,\ a_2+a_3\}$ &
\begin{tabular}{c}
$k_1 + k_2 = k_3$ \\
$k_2 + k_4 = k_1 + k_5$ \\
\end{tabular} \\ \hline
$\{a_1,\ a_2,\ a_1+a_2,\ a_1+a_3,\ a_1+a_2+a_3\}$ &
\begin{tabular}{c}
$k_1 + k_2 = k_3$ \\
$k_2 + k_4 = k_5$ \\
\end{tabular} \\ \hline
$\{a_1,\ a_1+a_2,\ a_1+a_3,\ a_2+a_3,\ a_1+a_2+a_3\}$ &
\begin{tabular}{c}
$k_1 + k_4 = k_5$ \\
$k_2 + k_3 = k_1 + k_5$ \\
\end{tabular} \\ \hline
$\{a_1,\ a_2,\ a_1+a_3,\ a_2+a_3,\ a_1+a_2+a_3\}$ &
\begin{tabular}{c}
$k_1 + k_4 = k_5$ \\
$k_2 + k_3 = k_5$ \\
\end{tabular} \\ \hline
\end{tabular}
\caption{Sets and equations for $|S| = 5$}
\end{center}
\end{table}

\begin{table}[htpb]
\begin{center}
\begin{tabular}{|c|c|} \hline
$A$-level set & $S$-level equations \\ \hline
$\{a_1,\ a_2,\ a_3,\ a_1+a_2,\ a_1+a_3,\ a_2+a_3\}$ &
\begin{tabular}{c}
$k_1 + k_2 = k_4$ \\
$k_1 + k_3 = k_5$ \\
$k_2 + k_3 = k_6$ \\
\end{tabular} \\ \hline
$\{a_1,\ a_2,\ a_3,\ a_1+a_3,\ a_2+a_3,\ a_1+a_2+a_3\}$ &
\begin{tabular}{c}
$k_1 + k_3 = k_4$ \\
$k_2 + k_3 = k_5$ \\
$k_1 + k_2 + k_3 = k_6$ \\
\end{tabular} \\ \hline
$\{a_1,\ a_2,\ a_1+a_2,\ a_1+a_3,\ a_2+a_3,\ a_1+a_2+a_3\}$ &
\begin{tabular}{c}
$k_1 + k_2 = k_3$ \\
$k_2 + k_4 = k_6$ \\
$k_1 + k_5 = k_6$ \\
\end{tabular} \\ \hline
\end{tabular}
\caption{Sets and equations for $|S| = 6$}
\end{center}
\end{table}

\begin{table}[htpb]
\begin{center}
\begin{tabular}{|c|c|} \hline
$A$-level set & $S$-level equations \\ \hline
$\{a_1,\ a_2,\ a_3,\ a_1+a_2,\ a_1+a_3,\ a_2+a_3,\ a_1+a_2+a_3\}$ &
\begin{tabular}{c}
$k_1 + k_2 = k_4$ \\
$k_1 + k_3 = k_5$ \\
$k_2 + k_3 = k_6$ \\
$k_1 + k_2 + k_3 = k_7$ \\
\end{tabular} \\ \hline
\end{tabular}
\caption{Sets and equations for $|S| = 7$}
\end{center}
\end{table}

\end{document}